\documentclass[12pt,reqno]{amsart}

\usepackage{epsfig}
\usepackage{amscd}
\usepackage[mathscr]{eucal}
\usepackage{amssymb}
\usepackage{amsxtra}
\usepackage{amsmath}
\usepackage[all]{xy}
\usepackage{mathtools}
\usepackage[pdfencoding=auto]{hyperref}
\usepackage{bookmark}
\usepackage{hyperref}

\DeclarePairedDelimiter\abs{\lvert}{\rvert}

\theoremstyle{plain}

\newtheorem{prop}{Proposition}[section]

\usepackage{amsthm}

\theoremstyle{definition}

\theoremstyle{remark}
\newtheorem{rem}[subsection]{Remark}

\newtheorem{ex}[subsection]{Example}

\begin{document}

\bigskip

\title[Dynamic approach for the zeros of zeta - Collisions and cancellation]{A dynamic approach for the zeros of the Riemann zeta function - Collision and repulsion}

\date{\today}

\author{Yochay Jerby}

\address{Yochay Jerby, Faculty of Sciences, Holon Institute of Technology, Holon, 5810201, Israel}

\email{yochayj@hit.ac.il}
%
%
% ...
\begin{abstract} 
For $N \in \mathbb{N}$ consider the $N$-th section of the approximate functional equation   
$$
\zeta_N(s)=  \sum_{n =1 }^N B_n(s),$$ where 
$$
B_n(s)= \frac{1}{2} \left [ n^{-s} + \chi(s) \cdot  n^{s-1} \right ].$$  
Our aim in this work is to introduce a new approach for the Riemann hypothesis by studying the way pairs of consecutive zeros of $\zeta_N(s)$ change with respect to $N$. 

For the initial stage, it is known that the non-trivial zeros of $\zeta_1(s)$ all lie on the critical line $Re(s)=\frac{1}{2}$. In the region $2N \leq Im(s) \leq 2 \pi (N+1)$ the function $\zeta_N(s)$ serves as an approximation of $\zeta(s)$ itself, and it was conjectured by Spira that in this region $\zeta_N(s)$ also admits zeros only on the critical line.  

We show that the appearance of zeros of a section off the critical line can be realized as the result of two consecutive zeros meeting and pushing each other off the critical line as $N$ changes, a process to which we refer to as \emph{a collision of zeros}. Based on a study of the properties of $\zeta_N(s)$, we suggest a way of re-arranging the order of summation of the elements $B_n(s)$ in $\zeta_{\left [ \frac{Im(s)}{2} \right ]}(s)$ that is expected to avoid collisions altogether, we refer to such a re-arrangement as a \emph{repelling re-arrangement}. In particular, establishing that the repelling re-arrangement indeed avoids collisions for \emph{any} pair of zeros would imply RH.

\end{abstract}

\maketitle
%
%
\section{Introduction and Summary of Main Results}

Let $\zeta(s)$ be the Riemann zeta function for $s=\sigma+i t$. In the range $ 1 < \sigma$ the function is defined as 
\begin{equation}
\label{eq:zeta-classic}
\zeta(s) := \sum_{n=1}^{\infty} \frac{1}{n^{s}},   
\end{equation} 
and is extended analytically to be defined on the whole complex plane. The Riemann Hypothesis postulates:

\bigskip

\hspace{-0.6cm} \bf The Riemann Hypothesis: \rm The non-trivial zeros $\rho$ of the Riemann zeta function, $\zeta(s)$, all lie on the critical line $\sigma=\frac{1}{2}$.

\bigskip

In fact, from the direct definition of $\zeta(s)$ itself, as given for instance in Eq. \ref{eq:zeta-classic}, it is rather hard to decipher any useful insight regarding the zeros $\rho$ and their location. In practice (beginning from Riemann himself), values of $\zeta(s)$ are not computed via the direct definition, but typically via some form of an approximate functional equation, see \cite{HL,HL2,HL3,Si}. In the 1960's Robert Spira conducted a theoretical and numerical study of the zeros of partial sums (sections) of the classical approximate functional equation, see \cite{SP1,SP2}. That is, Spira considered the zeros of the $N$-th sections of the approximate functional equation  
\begin{equation}
\label{eq:HL-AFE} 
\zeta_N(s)=  \sum_{n =1 }^N B_n(s),
\end{equation}  
where  
\begin{equation}
\label{eq:Bn} 
B_n(s)= \frac{1}{2} \left [ n^{-s} + \chi(s) \cdot n^{s-1} \right ],
\end{equation} 
and
\begin{equation} \chi(s):=2^s \pi^{s-1}sin \left ( \frac{\pi s}{2}\right ) \Gamma(1-s) \end{equation} is the 
function appearing in the functional equation 
\begin{equation}
\zeta(s) = \chi(s) \zeta (1-s).
\end{equation} For $N=1$ and $N=2$ (for $t$ sufficiently large) Spira showed that $\zeta_N(s)$ satisfy the Riemann hypothesis and admit zeros only on the critical line $\sigma = \frac{1}{2}$. Due to the 
functional equation, the function $\zeta_N(s)$ serves as an approximation of $\zeta(s)$ in the region $\sqrt{2 \pi N} \leq t \leq 2 \pi N$.
In fact, Spira conjectures the following RH for sections:

\bigskip

\hspace{-0.6cm} \bf Conjecture \rm (Spira \cite{SP1}): All the zeros of the section $\zeta_N(s)$ in the region $\sqrt{2 \pi N} \leq t \leq 2 \pi N$ lie on the critical line $\sigma =\frac{1}{2}$. 

\bigskip 

In his studies it is apparent that Spira mainly considered the properties of zeros of $\zeta_N(s)$ separately, for given $N$ at a time. Our aim in this work is to present a new approach to Spira's conjecture. The main feature is that rather than studying the zeros of a given section $\zeta_N(s)$ independently, we are interested in studying the way the zeros of the sections $\zeta_N(s)$ change with respect to $N$. Concretely, let us summarize the main points of our approach:   

\begin{enumerate}

\item Rouche's theorem implies the existence of a one-to-one correspondence between the zeros of $\zeta_N(s)$ and $\zeta_{N+1}(s)$ in the critical strip. In particular, no "new" zeros are created in the critical strip during the transition from one section to the other. We are thus interested in the way in which the zeros of the sections $\zeta_N(s)$ change as $N$ changes from the first section $\zeta_1(s)$, whose zeros are known to lie on the critical line, to the $\left [ \frac{t}{2} \right ]$-th section, whose zeros are expected to lie on the critical line by Spira's conjecture. 

\item Even-though the zeros are known to begin at the initial stage of $N=1$ on the critical line, and are expected to eventually also lie on the critical line for $N=\left [ \frac{t}{2} \right ]$, some of the intermediary sections $\zeta_N(s)$ might violate RH. That is their zeros do not nessecerally need to lie on the critical line for any $1 \leq N \leq \left [ \frac{t}{2} \right ]$. We observe, however, that zeros of $\zeta_N(s)$ can appear off the critical line only if a process to which we refer as \emph{collision} occurred between a pair of consecutive zeros in a previous stage. This leads us to study the interactions between pairs of zeros as $N$ changes.        

\item The question is thus, could collisions between a given pair of zeros be avoided? By definition, the section $\zeta_{\left [ \frac{t}{2} \right ]} (s)$ of interest is given as the sum of the elements 
\begin{equation} 
B_n(s) = \frac{1}{2} \left [ n^{-s} + \chi(s) \cdot n^{s-1} \right ]
\end{equation}  
for the first $n=1,..., \left [\frac{t}{2} \right ]$. The initial approach described above is based on summing the elements $B_n(s)$ in consecutive order (giving rise to the sections $\zeta_N(s)$) until the $\left [\frac{t}{2} \right ]$-th element is added. However, the problem with this naive approach is that the unwanted collisions could occur between various pairs of zeros, pushing those zeros away from the critical line at certain stages. 

The question can thus be rephrased as follows: can the summation order of the elements $B_n(s)$ be re-arranged in a different manner so that collisions would be avoided altogether? That is, is it possible that the collisions are not an essential phenomena but rather a by-product of the specific consecutive order of summation considered? 

We give various theoretical justifications suggesting that the answer to this question is affirmative. In fact, based on a few fundamental observations regarding the way the sections $\zeta_N(s)$ change with respect to $N$, we suggest a well-defined re-arrangement of the order of summation. We conjecture that this re-arrangement avoids collisions for any pair of consecutive zeros altogether, and hence refer to it as a \emph{repelling re-arrangement}. 

\end{enumerate}

In short our our observations could be summarized as follows:

\begin{itemize} 
\item The non-trivial zeros of $\zeta(s)$ are "born" as zeros of $\zeta_1(s)$ on the critical line, which are regulated and well-understood, and dynamically develop by gradually adding the first $\left [ \frac{t}{2} \right ]$ elements of $B_n(s)$ to $\zeta_1(s)$ to obtain $\zeta_{\left [ \frac{t}{2} \right]}(s)$. 

\item We \bf conjecture \rm that, when the addition is done via the repelling re-arrangement, no collision occurs between consecutive zeros and, hence, the zeros always remain on the critical line. This includes the final stage, where their position \emph{on the critical line} is identical to that of the zeros of $\zeta(s)$, up to a negligible error. In particular, the non-trivial zeros of $\zeta(s)$ must lie on the critical line, that is, they satisfy RH.
\end{itemize}
   
Let us note that throughout the work we would also be concerned with a variant of the classical sections $\zeta_N(s)$ to which we refer as the accelerated sections $\widetilde{\zeta}_N(s)$. The accelerated sections are given as partial sums of the Euler transformation of series of the defining sum of $\zeta(s)$ given by Eq. \ref{eq:zeta-classic}. Everything mentioned for the classical sections applies to the accelerated sections as well. However, the advantage of the accelerated sections $\widetilde{\zeta}_N(s)$ is twofold. First, the approximation of zeta afforded by them is far superior to that given by the classical sections $\zeta_N(s)$. Moreover, their change with respect to $N$ is "smoothened" relative to that of the classical sections.  

\bigskip 

The rest of this work is devoted to explaining and expanding in detail on Points (1)-(3) and is organized as follows: In Section \ref{s:1} we introduce the Euler transformation of series procedure for $\zeta(s)$ and the corresponding accelerated sections $\widetilde{\zeta}_N(s)$. An initial discussion on the analytical distinctions between the classical and accelerated sections is presented. In Section \ref{s:2} we discuss the zeros of $\zeta_1(s)$, review their recent representation in terms of the Lambert function due to Franca-LeClair given in \cite{FL,FL2} as well as present an alternative new description. In Section \ref{s:3} we study the zeros of the sections $\zeta_N(s)$ and give examples for their collisions as $N$ changes from zero to $\left [ \frac{t}{2} \right ]$. Furthermore, new interpretations of Gram's law are presented, see Remark \ref{rem:gram}. In Section \ref{s:4}, based on observations from Section \ref{s:1} and Section \ref{s:3}, we introduce the repelling re-arrangement and illustrate how it leads to avoiding of collisions. We consider the Davenport-Heilbronn function $\mathcal{D}(s)$, and show how the various phenomena discussed for zeta are violated for $\mathcal{D}(s)$, see Remark \ref{rem:DH}. Relations to the Montgomery pair correlation conjecture is also considered, see Remark \ref{rem:MPC}. Finally, in Section \ref{s:5} we present a summary and concluding remarks.

 \section{The Euler transformation of series for $\zeta(s)$ and its comparison to the classical sum} 
 \label{s:1}

 For $ 1 < \sigma $ the Riemann zeta is defined by the series 
\begin{equation} 
 \zeta(s) = \sum_{n=1}^{\infty} \frac{1}{n^s}.
 \end{equation} Although that for $\sigma \leq 1$ the series is not converging it is nevertheless interesting to consider how it behaves in this region. Consider the partial sums 
 \begin{equation} 
 S_N(s):= \sum_{n=1}^N \frac{1}{n^s}
 \end{equation}
 for $N \in \mathbb{N}$. The following Fig. \ref{fig:f1} shows the values of $ln \abs{S_N(\frac{1}{2}+17500 i)}$ for $N=1,...,5000$ (blue) together with the value $ln \abs{\zeta(\frac{1}{2}+17500 i)}$ (orange):
  
\begin{figure}[ht!]
	\centering
		\includegraphics[scale=0.35]{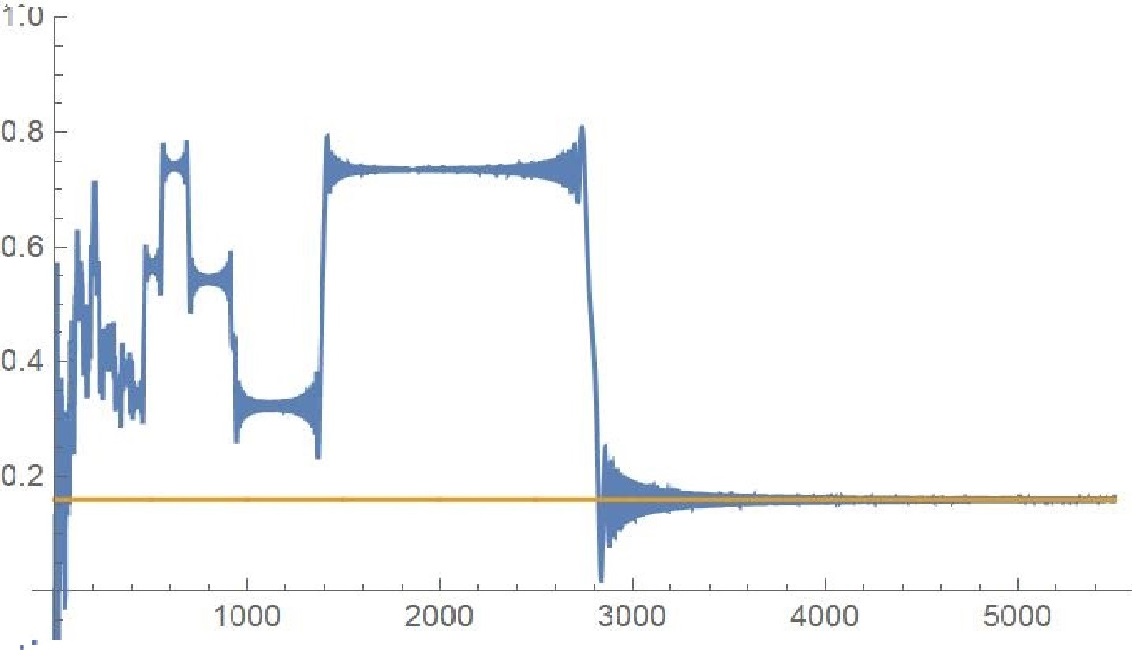} 
	\caption{The values of $ln \abs{S_N(\frac{1}{2}+17500 i)}$ for $N=1,...,5000$ (blue) together with the value $ln \abs{\zeta(\frac{1}{2}+17500 i)}$ (orange).}
\label{fig:f1}
	\end{figure}

Figure 1 could be considered as illustrating the typical behaviour of the partial sums $S_N(s)$. In particular, one can see from Fig. \ref{fig:f1} two features:

\begin{enumerate}

\item Although the series is non-convergent, for $N>>0$ big enough\footnote{It should be noted that if $N$ is taken too big $\zeta_N(s)$ what start deviating from the value of $\zeta(s)$.} the partial sums $S_N(s)$ do eventually serve as approximations of $\zeta(s)$. 

\item The partial sums are $S_N(s)$ seen to fluctuate around other values before stabilizing around the final value of $\zeta(s)$ for $N$ big enough. 
\end{enumerate} 

In fact, the phenomena presented in Fig. 1 is explained by the classical approximate functional equation of Hardy and Littlewood. The classical approximate functional  equation for the Riemann zeta function was proven by Hardy and Littlewood in the series of works \cite{HL,HL2,HL3}. For $0  \leq \sigma \leq 1$, the theorem states that the following holds 
\begin{equation}
\label{eq:HL-AFE} 
\zeta(s)= \sum_{n \leq x } n^{-s} + \chi(s)  \left ( \sum_{n \leq y } n^{s-1} \right )  + O \left ( x^{-\sigma} + x^{\frac{1}{2} - \sigma} y^{-\frac{1}{2}}  \right ),
\end{equation} when $x,y \geq 1$ are such that $\vert t \vert = 2 \pi x y$. In particular, the AFE explains what are the sub-values around which $S_N(s)$ fluctuates before stabilizing around $\zeta(s)$. These are exactly the values $\zeta(s)-\chi(s) \left ( \sum_{n=1}^M n^{s-1} \right )$ and the region of values of $N$ for which these fluctuations occurs is roughly the region 
\begin{equation}
\left [ \frac{t}{2 (M+1) \pi} \right ] \leq N \leq \left [ \frac{t}{2 M \pi} \right ].
\end{equation} 
This rephrasing of the AFE in view of the behaviour expressed in Fig. \ref{fig:f1} would be of importance in the discussion of the following Section \ref{s:5}. 

Let us now turn to consider the Euler transformation of series for $\zeta(s) = \sum_{n=1}^{\infty} \frac{1}{n^s}$. Recall that to a given alternating series (convergent or divergent) 
\begin{equation}
S=\sum_{n=1}^{\infty} (-1)^{n-1} a_n 
\end{equation}
  one can apply the highly classical procedure of Euler's transformation of series, see \cite{Eu,Har,K}. Concretely, one can re-write the series as   
\begin{equation} 
S= \sum_{n=0}^{\infty} \frac{\Delta^n a_1}{2^{n+1}}, 
\end{equation}    
where 
\begin{equation} 
\Delta^n a_1:=\sum_{k=0}^n (-1)^k \binom{n}{k} a_{k+1}.   
\end{equation}  
As a direct application of the transformation to the defining series of zeta, $\zeta(s)=\sum_{n=1}^{\infty} \frac{1}{n^s}$, itself, that is to the sequence 
\begin{equation}
a_n(s):= \frac{(-1)^{n-1}}{n^s},
\end{equation}
 we obtain the following global formula\footnote{Formula Eq. \ref{eq:Euler-acc} is a variant of the Hasse-Sondow global formula, obtained by applying Euler acceleration to the defining series of the Dirichlet eta function $\eta(s)$ instead of Eq. \ref{eq:zeta-classic}, see \cite{H,S}.}: 

\begin{prop}
 The following formula  
\begin{equation} 
\label{eq:Euler-acc}
\zeta(s)=  \sum_{n=0}^{\infty} \widetilde{A}(s,n), 
\end{equation} 
where 
\begin{equation}
\widetilde{A}(s,n):=  \frac{1}{2^{n+1}} \sum_{k=0}^n \binom{n}{k} \frac{1}{(k+1)^s}, 
\end{equation} 
holds for any $s \in \mathbb{C}$. 
\end{prop}  

Let us compare the behaviour of the transformed series to that of the classical one. Denote the partial sums of the transformed series by 
\begin{equation} 
\widetilde{S}_N(s) = \sum_{n=0}^N \widetilde{A}(s,n). 
\end{equation} 
 for $N \in \mathbb{N}$. Consider the following Fig. \ref{fig:f2}:
 
  \begin{figure}[ht!]
	\centering
		\includegraphics[scale=0.45]{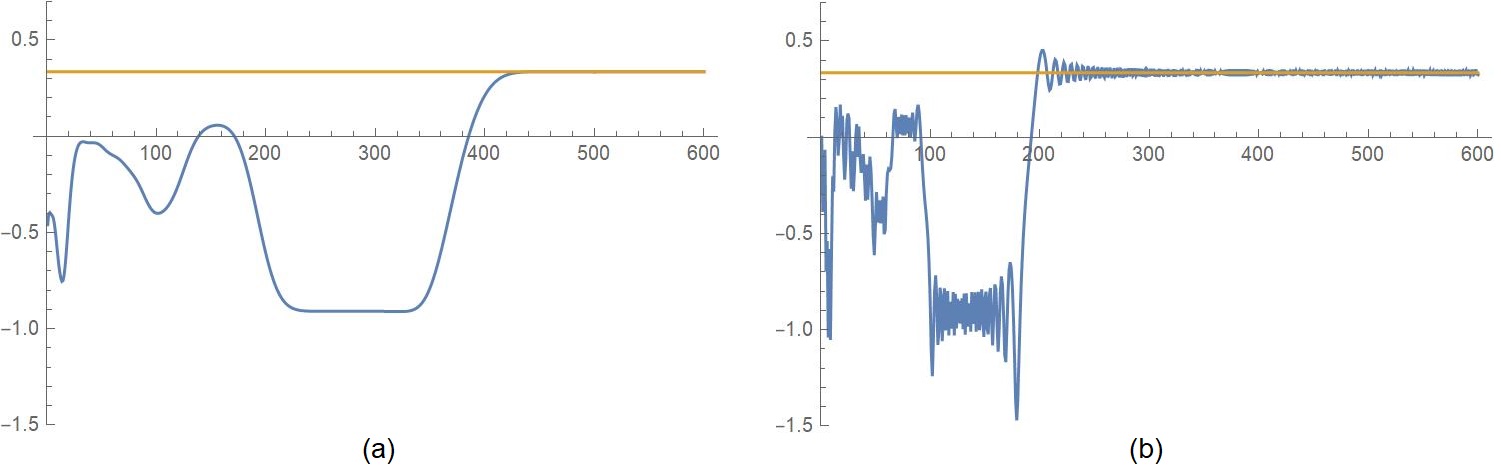} 
	\caption{(a) Values of $ln \abs{\widetilde{S}_N(\frac{1}{2}+1200 i)}$ for $N=1,...,600$ (blue) together with the value $ln \abs{\zeta(\frac{1}{2}+1200 i)}$ (orange) (b) Values of $ln \abs{S_N(\frac{1}{2}+1200 i)}$ for $N=1,...,600$ (blue) together with the value $ln \abs{\zeta(\frac{1}{2}+1200 i)}$ (orange).}
\label{fig:f2}
	\end{figure}

In Fig. \ref{fig:f2} one sees the following two features showing the advantage of the accelerated series over the classical one: 

\begin{enumerate}
\item Contrary to the classical case, the transformed partial sums $\widetilde{S}_N(s)$ seem to be a superb approximation of 
$\zeta(s)$ for $N>>0$ big enough. In fact, in the range $2N \leq t \leq 2 (N+1)$ one has the following approximation with exponentially decaying error term
\begin{equation} 
\zeta(s)=\widetilde{S}_N(s)+O(e^{-\omega \abs{t}})
\end{equation}  for $\omega >0$ is a certain positive constant, which is a variant of our recent result in the setting of the Hasse-Sondow formula, see \cite{J}. This approximation, whose error term is of exponentially decaying error, is far superior to the approximation afforded by the classical approximate functional equation, whose error term is only algebraic. 
\item Although the application of the transformation seems to \emph{smoothen} the behaviour of $\widetilde{S}_N(s)$ with respect to $N$, it 
still bares much overall similarity to the behaviour of the original classical sums $S_N(s)$. In fact, this smoothing feature could 
be explained by noting that by changing the order of summation one can also write  
\begin{equation}
\label{eq:global}
\overline{S}_N(s) = \sum_{k=0}^{N} \widetilde{a}(k,N) (k+1)^{-s}. 
\end{equation} 
where the constants are given by 
\begin{equation} 
\widetilde{a}(k,N):= \sum_{n=k}^N \frac{1}{2^{n+1}} \binom{n}{k}.
\end{equation} 
In comparison, the classical sections could be written via a similar formula with $a(k,N)=1$. In other words, the transformed sections 
$\widetilde{S}_N(s)$ could be considered as adding the weights $\widetilde{a}(k,N)$ instead of the trivial weights for the classical sections. 
\end{enumerate}

In particular, let us define the accelerated $N$-th sections of the global representation 
\begin{equation}
\label{eq:ACC-AFE} 
\widetilde{\zeta}_N(s)= \sum_{n =0 }^N \widetilde{B}_n(s),
\end{equation} 
where 
\begin{equation}
\label{eq:ACC-AFE} 
\widetilde{B}_n(s)= \frac{1}{2} \left [  \widetilde{A}(s,n) + \chi(s) \cdot  \widetilde{A}(1-s,n) \right ].
\end{equation}  
 In what follows we would typically compare the properties of the classical sections $\zeta_N(s)$ to those of the accelerated ones $\widetilde{\zeta}_N(s)$.  
 \section{On the zeros of $\zeta_1(s)$ and their representations} 
 \label{s:2} 
   In this section we consider the zeros of the first section $\zeta_1(s)= 2 \widetilde{\zeta}_0(s)= 1+ \chi(s)$. The first part of this section contains review of known results of Spira and Franca-LeClair, see \cite{FL,SP1,SP2}. The zeros of $\zeta_1(s)$ were studied in the 1960's by Spira who showed that all the zeros in the critical strip must lie on the critical line. Moreover, Spira also observed that between any two Gram points\footnote{It should be noted that Spira conducted his studies before the establishment of the current notations of the Lambert function.} lies a zeros of $\zeta_1(s)$, and vice versa, a feature to which we shall return shortly. 
   
  Consider the following Fig. \ref{fig:f3} which shows $ln \abs{\zeta \left ( \frac{1}{2} +it \right ) }$ (blue) and $ln \abs{\zeta_1 \left ( \frac{1}{2} +it \right ) }$ (orange) for $0 \leq t \leq 50$:
\begin{figure}[ht!]
	\centering
		\includegraphics[scale=0.4]{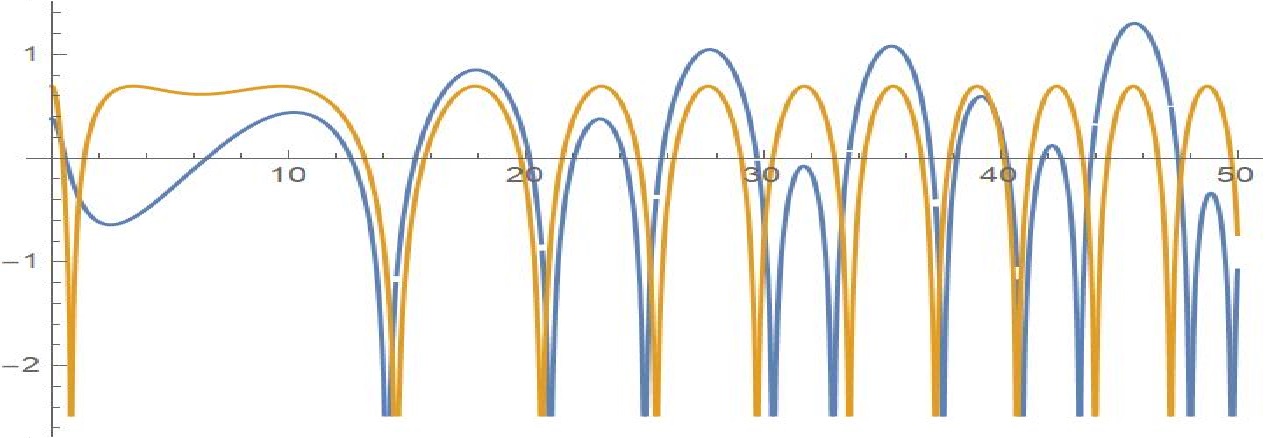} 
	\caption{Graphs of $ln \abs{\zeta \left ( \frac{1}{2} +it \right ) }$ (blue) and $ln \abs{\zeta_1 \left ( \frac{1}{2} +it \right ) }$ (orange) for $0 \leq t \leq 50$.}
\label{fig:f3}
	\end{figure}
	
As one can see, the zeros of $\zeta_1(s)=1+\chi(s)$ already serve as an initial crude approximation of the zeros of $\zeta(s)$ on the critical line, in the prescribed region. Recently, the zeros of $1+ \chi(s)$ were further studied (to a certain extent re-discovered) by Franca and LeClair in \cite{FL,FL2}, where they showed that they can be described in terms of the Lambert function, see \cite{CGHJK,Hay}.

For large $t$ we have by Stirling's formula 
\begin{equation} 
\label{eq:stir} 
\Gamma(\sigma+it)=\sqrt{\frac{2 \pi}{\sigma+it}} \left ( \frac{\sigma+it}{e} \right )^{\sigma+it} \left ( 1+ O \left ( \frac{1}{t} \right ) \right ).
\end{equation} 
It also follows that 
\begin{equation}
\label{eq:chi-asymp}
\chi(\sigma+it)=\left ( \frac{2 \pi}{t} \right )^{\sigma+it -\frac{1}{2}} e^{i (t+\frac{\pi}{4}) } \left ( 1+ O \left (\frac{1}{t} \right ) \right ) .
\end{equation}
Hence, define
\begin{equation}
\label{eq:chi-overline}
\overline{\chi}(\sigma,t) :=\left ( \frac{2 \pi}{t} \right )^{\sigma+it -\frac{1}{2}} e^{i (t+\frac{\pi}{4}) }.
\end{equation} 
Consider the equation 
\begin{equation} 
1+\overline{\chi}(\sigma,t) =1+\left ( \frac{2 \pi}{t} \right )^{\sigma+it -\frac{1}{2}} e^{i (t+\frac{\pi}{4}) }=0.
\end{equation} 
Taking absolute value implies $\sigma = \frac{1}{2}$ while taking argument implies the equation 
\begin{equation}
\label{eq:FL1}
 \frac{t}{2 \pi} \cdot ln \left (\frac{t}{2\pi e } \right ) = n-\frac{11}{8}.
\end{equation}
 Recall that the Lambert function is a multivalued function given by the branches of the inverse function of $we^w$. For each $r$ there is one branch, denoted $W_r(z)$ such that 
\begin{equation} 
W_r(z) e^{W_r(z)}=z. 
\end{equation} 
Over the real numbers only the two branches $W_0(x)$ (also called the principal branch) and $W_{-1}(x)$ are required. From Eq. \ref{eq:FL1} Franca and LeClair deduced that the zeros of $1+ \chi(s)$ can be approximated by $\rho^0_n = \frac{1}{2}+i t^0_n$ where $t^0_n$ are the solutions of 
\begin{equation} 
\label{eq:FLsol2}
t^0_n = \frac{(8n-11) \pi  }{4 W_0( \frac{8n-11}{8e} ) }. 
\end{equation}
More generally, let us consider the functions 
$$ B_k(s):= k^{-s}+ \chi(s) \cdot k^{s-1}.$$ 
By a similar argument the solutions of the equation $B_k(s)=0$ lie on the critical line and their imaginary parts are given by the following equation  
$$ \frac{t}{2 \pi} \cdot ln \left ( \frac{t}{2 \pi k^2 e} \right )= n -\frac{11}{8}. $$ 
Let us define 
$$ \begin{array}{ccc} \widetilde{t}^+_k(m) =  \frac{(8m-3)\pi}{4 W_{0} \left ( \frac{8m-3}{8 k^2 e} \right ) } & ; & \widetilde{t}^-_k(m)=\frac{(5-8m)\pi}{4 W_{-1} \left ( \frac{5-8m}{8 k^2 e} \right ) }. \end{array} $$
We have the following description of the zeros of $B_k(s)$ on the critical line, generalizing the Franca-LeClair formula for the case $k=1$: 

\begin{prop} 
\label{prop:2.1}
Let $s^k_m=\frac{1}{2}+i t^k_m$ be the zeros of $B_{k}(s)$ on the critical line. Then $t^k_m=-\overline{t}^k_{-m}$ and for the zeros in the upper half plane with $t^k_m>0$ the following holds: 
\begin{enumerate} 

\item $t^k_{m} = \widetilde{t}^-_k(m)$ for $1 \leq m \leq k^2$.

\item $t^k_{m} = \widetilde{t}^+_k(m-2 k^2)$ for $k^2+1 \leq m$.
\end{enumerate}
\end{prop} 

As an illustration, consider for instance the following Fig. \ref{fig:f4} which shows the graph of $ln \abs{B_3 \left ( \frac{1}{2} +it \right ) }$ for $0 \leq t \leq 150$:
\begin{figure}[ht!]
	\centering
		\includegraphics[scale=0.4]{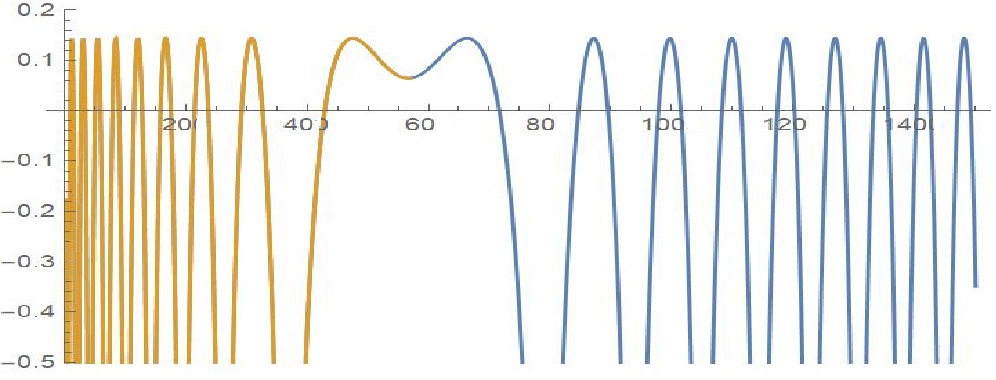} 
	\caption{ Graph of $ln \abs{\widetilde{B}_3 \left ( \frac{1}{2} +it \right ) }$ for $0 \leq t \leq 150$.}
\label{fig:f4}
	\end{figure}
	 In Fig. \ref{fig:f4} the region of the first \emph{negative} zeros $t^k_m=\widetilde{t}^-_k(m)$ for $m=1,...,9$ is marked in orange and the region of the \emph{positive} zeros $t^k_m=\widetilde{t}^+_k(m-2k^2)$ is marked in blue. In view of the above, we also obtain the following alternative description of the zeros of $B_k(s)$ in the relevant region, which does not involve the Lambert function: 
\begin{prop}[alternative description of the zeros of $B_k(s)$]
The zeros of $B_k(s)$ in the region $2 N  \leq t \leq 2(N+1) $ are given by $$\widetilde{T}_j^k  =\frac{ 8 N+ 8\pi \cdot j - 11 \pi}{4 \left (ln(N)-ln(\pi \cdot k^2) \right )}$$
for
\begin{equation} 
\frac{N}{\pi}  ln \left ( \frac{N}{\pi \cdot k^2 e} \right )+ \frac{11 }{8} \leq  j \leq \frac{N}{\pi} ln \left (\frac{N}{\pi \cdot k^2 e} \right ) +\frac{1}{\pi} ln \left (\frac{N}{\pi \cdot k^2} \right ) + \frac{11}{8}
\end{equation}
\end{prop}    
\begin{proof} In view of Eq. \ref{eq:chi-overline}, in the region $2 N  \leq t \leq 2(N+1) $ the following approximation holds  
\begin{equation}
 ln  \left| B_k \left (\frac{1}{2} +it \right ) \right| \approx \left| \frac{2}{ \sqrt{k}}  sin \left (\frac{t}{2} ln \left ( \frac{N}{\pi k^2} \right ) + \frac{11 \pi }{8}-N  \right ) \right|.
 \end{equation}
The zeros of the approximating function are given by 
\begin{equation}
\widetilde{T}_j^k  =\frac{ 8 N+ 8\pi \cdot j - 11 \pi}{4 \left (ln(N)-ln(\pi \cdot k^2) \right )}
\end{equation}
for $j \in \mathbb{Z}$. 
\end{proof}  

In particular, let us define the distance between two consecutive zeros of $B_k(s)$ in the region $2N\leq t \leq 2(N+1)$ by 
\begin{equation}
\label{eq:dist}
d(N,k):= \left|\frac{ 2\pi }{ln(N)-ln(\pi \cdot k^2)} \right|
\end{equation}
Let us conclude this section with the following two remarks:

	 \begin{rem}[The Franca-LeClair approach for RH]
	 In \cite{FL,FL2} Franca and LeClair further (formally) introduce the equation 
 \begin{equation}
 \label{eq:FL4} 
\frac{t}{2 \pi} \cdot ln \left ( \frac{t}{2 \pi e} \right ) +\frac{1}{\pi} lim_{\delta \rightarrow 0^+} \left ( arg \left ( \zeta \left ( \frac{1}{2} + \delta +it \right ) \right ) \right ) =n- \frac{11}{8}.
\end{equation}
Note that the leading term of this equation is 
\begin{equation}
\label{eq:FL1}
 \frac{t}{2 \pi} \cdot ln \left (\frac{t}{2\pi e } \right ) = n-\frac{11}{8},
\end{equation}
which admit the unique solution $t_n^0$ for any $n$, as defined in Eq. \ref{eq:FLsol2}. In contrast, the full Eq. \ref{eq:FL4} is only formally defined due to the fact that the term  
\begin{equation}
\label{eq:FL-arg}
\frac{1}{\pi} lim_{\delta \rightarrow 0^+} \left ( arg \left ( \zeta \left ( \frac{1}{2} + \delta +it \right ) \right ) \right )
\end{equation} is not nessecerally well defined for all zeros of $\zeta(s)$. Franca and LeClair show that if Eq. \ref{eq:FL-arg} is well defined then equation Eq. \ref{eq:FL1} has a solution $t_n$ for any $n$, and to any such solution corresponds a zero $\rho_n=\frac{1}{2}+it_n$ of $\zeta(s)$ on the critical line. In particular, it is shown that the RH is equivalent to the question of the well-defindness of the term Eq. \ref{eq:FL-arg}.

 It should be noted that the question of the well-defindness of the term Eq. \ref{eq:FL-arg} is a highly non-trivial and elusive matter by itself and, as for RH, to which it is equivalent, any substantial reason for the well-defindness of the term Eq. \ref{eq:FL-arg} is currently lacking. Moreover, in \cite{FL,FL2} Franca and LeClair also study the analogous equation for the Davenport-Heilbronn function $\mathcal{D}(s)$, which is an $L$-function satisfying a functional equation but for which the RH fails, and as a result for which the analogous version of Eq. \ref{eq:FL-arg} is not well defined for all zeros. The results of \cite{FL,FL2}, however, do not offer insight on the reason for the difference between $\zeta(s)$ and $\mathcal{D}(s)$.    

 On the other hand, the results of Franca-LeClair do imply the important fact that the RH is equivalent to showing that to any zero of $1+\chi(s)$ corresponds a unique zero of $\zeta(s)$ on the critical line, whose imaginary part is given as a solution of Eq. \ref{eq:FL4}. In the setting of Franca-LeClair, the relation between the zeros of the functions is a formal matter of adding the term in Eq. \ref{eq:FL-arg} (which is a-priori not well defined) to Eq. \ref{eq:FL1}. However, again, the results of \cite{FL,FL2} do not give insight as of why such a relation between the zeros of $\chi(s)$ and $\zeta(s)$ should exist in practice.   
	 \end{rem}
 
\begin{rem}[Gram points] \label{rem:gram} Recall that the Riemann-Siegel theta function is defined by 
\begin{equation} 
\nu(t) = arg \left ( \Gamma \left ( \frac{1}{4} + \frac{i t}{2} \right ) \right ) -\frac{t}{2} ln(t).  
\end{equation} 
The $n$-th Gram point is given as the unique solution of the equation $
\nu(g_n) = \pi n$, see for instance \cite{E}. The Gram points could be approximated by 
\begin{equation} 
g_n \approx \frac{(8n+1) \pi  }{4 W_0( \frac{8n+1}{8e} ) }. 
\end{equation}  
From this, and the definition of $t^0_n$ in Eq. \ref{eq:FLsol2} Spira's observation that a Gram point $g_n$ is always found between any two consecutive elements of $t^0_n$ (and vice versa), follows immediately. 

The first few Gram points were computed by Gram in \cite{G} where he also observed that \emph{typically} one has "Gram's law": 
\begin{equation} 
\label{eq:gram}
Re \left ( \zeta \left ( \frac{1}{2} +it \right ) \right ) = (-1)^n Z(g_n) >0, 
\end{equation} 
where $
Z(t):= e^{i \nu (t)} \zeta \left ( \frac{1}{2}+it \right )$ is the Riemann-Siegel function. As $Z(t)$ is a real function, whenever Eq. \ref{eq:gram} is satisfied for two consecutive Gram points, $g_n$ and $g_{n+1}$, it implies the existence of a zero of zeta between these two point, on the critical line. However, in \cite{H}
Hutchinson computed the Gram points and the values $Z(g_n)$ up to $n=138$. In particular, Hutchinson found examples in which violations of Eq. \ref{eq:gram} occurs, the first such example occurring for $n=126$. In the next Section \ref{s:3} we will present a new interpretation of Gram's law in terms of the zeros of $\zeta_N(s)$. 
\end{rem} 

 \section{On the change of zeros of $\zeta_N(s)$ with respect to $N$ and their collisions} 
 \label{s:3} 

In the previous section we described the zeros of $\zeta_1(s)=2\widetilde{\zeta}_0(s)$ on the critical strip, and saw that they all lie on the critical line. In this section we are interested in studying how the zeros of $\zeta_N(s)$ (or $\widetilde{\zeta}_N(s)$\footnote{All arguments of this section apply equally well for the accelerated sections $\widetilde{\zeta}_N(s)$.}) change with respect to $N$. 

First recall that by Rouche's theorem, for any two complex-valued functions $F(s)$ and $G(s)$ holomorphic inside some region $K \subset \mathbb{C}$ with closed contour $\partial K$, if $\abs{G(s)} < \abs{F(s)}$ on $\partial K$, then $F(s)$ and $F(s) + G(s)$ have the same number of zeros, with multiplicity, inside $K$. In our case, Let us set 
 \begin{equation}
 \begin{array}{ccc} F_N(s)=\zeta_1(s) & ; & G_N(s)=\sum_{n=2}^N  B_n(s), \end{array}
\end{equation} 
such that $\zeta_N(s)=F_N(s)+G_N(s)$, by definition. It is easy to show that in compact regions with $\abs{\sigma}>>0$ big enough, the condition
$\abs{G_N(s)} < \abs{F_N(s)}$ is satisfied. Hence, we have: 
\begin{itemize}
\item The zeros of $\zeta_1(s)$ in the critical strip are in one-to-one correspondence with the zeros of the sections $\zeta_N(s)$, for any $N$, and no new zeros in the critical strip can be created in the transition from $\zeta_1(s)$ to $\zeta_N(s)$.
 \end{itemize} 

 Moreover, it is expected that the addition of the error term of the AFE (especially in the transformed case) would also not create new zeros in the transition from $\widetilde{\zeta}_N(s)$ to $\zeta(s)$ in the region $2N \leq t \leq 2(N+1)$. Hence we get: 
 \begin{itemize}
\item The non-trivial zeros of the Riemann zeta function $\zeta(s)$ on the critical strip are in one-to-one correspondence with the zeros of the section $\zeta_N(s)$ for $N=\left [ \frac{t}{2} \right ]$ and the distance between them is exponentially small with respect to the size of $t$.
 \end{itemize} 
 
 It should be noted that this argument does not yet imply anything regarding the location of the zeros of $\zeta(s)$, but only on the existence of a one-to-one correspondence between the non-trivial zeros of $\zeta(s)$ and those of the sections $\zeta_N(s)$ from zero to $\left [ \frac{t}{2} \right ]$. In summary, in view of the above remark, we obtain: 
\begin{itemize} 
\item The RH would follow from showing that to any zero of $\widetilde{\zeta}_0(s)$ on the critical line corresponds a non-trivial zero of $\zeta_N(s))$ for $N=\left [ \frac{t}{2} \right ]$ lying \emph{on the critical line}. 
\end{itemize}
In other words
\begin{itemize}
\item Spira's RH for sections $\Rightarrow$ RH.
\end{itemize}
Before proceeding let us make the following remark: 
\begin{rem}[Comparison to the approach of Franca-LeClair] As mentioned in the previous section, in the setting 
of Franca-LeClair the zeros of $1+\chi(s)$ arise as leading term approximation of the full Franca-LeClair equation (Eq. \ref{eq:FL4}), which should theoretically coincide (assuming RH) with the zeros of $\zeta(s)$ on the critical line. In particular, this suggested theoretical relation is the result of the algebraic fact that Eq. \ref{eq:FL1} is the first-order approximation of Eq. \ref{eq:FL4}. In \cite{FL,FL2} Franca and LeClair also presented extensive numerical evidence for various connections between the properties of 
the zeros of $1+\chi(s)$ and those of the known zeros of $\zeta(s)$. However, Franca and LeClair do not suggest further formal explanations for why a connection between the two collections exists.  

In our setting of the AFE the collection of zeros of $1+\chi(s)$ arises as the zeros of the zero-th section $\widetilde{\zeta}_0(s)$. However, in our setting we also get the corresponding collection zeros of the section $\widetilde{\zeta}_N(s)$, for any other $N \in \mathbb{N}$. In this sense the zeros of $\widetilde{\zeta}_N(s)$ could be thought of as a generalization of the Franca-LeClair approximations for any $N$. In particular, we view these collections as forming a gradual "bridge" between the zeros of $1+\chi(s)$ and the actual zeros of $\zeta(s)$ essentially attained for $N=\left [ \frac{t}{2} \right ]$, which gives a satisfying explanation for why the relation between the FL zeros and those of $\zeta(s)$ exists.    
\end{rem}
 
 We are hence interested in studying the way the zeros change with respect to $N$. As the parameter $N$ is discrete, we need to make sense of what we mean by "change with respect to $N$". For any $N \in \mathbb{N}$ define 
 \begin{equation}
 \zeta^t_N(s) = (1-t) \cdot \zeta_N(s)+t \cdot \zeta_{N+1}(s),
 \end{equation} 
 for $0 \leq t \leq 1$. In particular, $\zeta^0_N(s) =\zeta_N(s)$ and $\zeta^1_N(s)=\zeta_{N+1}(s)$. Let us consider the way the zeros of $\zeta_N^t(s)$ change with respect to $t$. Assume the zeros of $\zeta_N(s)$ in the critical strip are given by $\rho^N_n$ for $n \in \mathbb{Z}$. For any $n \in \mathbb{Z}$ we can define the continuous family of zeros $\rho_n^N(t)$ of $\zeta^t_N(s)$. This can be done only as long as no double zeros occur. If a double zero occurs we refer to such an instance as a "collision" between zeros. Starting from $\zeta_1(s)$ we can thus inductively define $\rho_n^N$ for any $\zeta_N(s)$ (similarly $\widetilde{\rho}_n^N$ for any $\widetilde{\zeta}_N(s)$) as follows: 
\begin{enumerate} 
 \item As long as no collisions occur define inductively $\rho_n^{N+1}=\rho_n^N(1)$. The main feature is that collisions can occur only between two consecutive zeros $\rho_n^N$ and $\rho_{n+1}^N$ and as long as no collisions occur the zeros \underline{must remain on the critical line}. 
 \item If a collision between two consecutive zeros $\rho_n^N(t)$ and $\rho^N_{n+1}(t)$ on the critical line occurs for some $0 \leq t_0 \leq 1$ then the two zeros get "pushed off the critical line" in a symmetric manner along the critical line, so that we can continue to define $\rho_n^N(t)$ and $\rho^N_{n+1}(t)$ for $t_0 \leq t \leq 1$ such that $\rho_n^N(t)=1-\rho^N_{n+1}(t)$.
\item After two consecutive zeros collided, for some $N$ and $t$, the only way a collision can occur again is if the two zeros return to the critical line. This is because the only two zeros symmetric along the critical line can meet along the critical line.  
  \end{enumerate}

Let us consider the following example:

\begin{ex}[Collisions for $\widetilde{\zeta}_N(s)$] The following Fig. \ref{fig:f5} illustrates the phenomena of collisions by showing the graphs of $ln \abs{\zeta(\frac{1}{2}+it)}$ (blue) and
$ln \abs{\widetilde{\zeta}_N(\frac{1}{2}+it)}$ for $N=8,9$ (orange) in the range $86 \leq t \leq 90$: 
      
       \begin{figure}[ht!]
	\centering
		\includegraphics[scale=0.35]{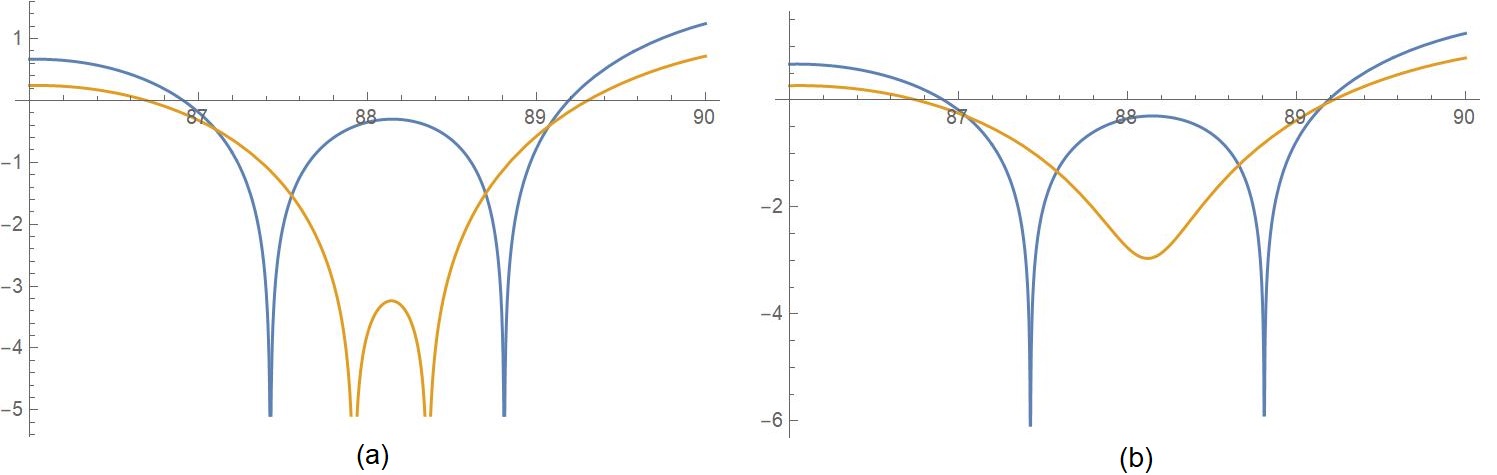} 
	\caption{Graphs of $ln \abs{\zeta(\frac{1}{2}+it)}$ (blue) and
$ln \abs{\widetilde{\zeta}_N(\frac{1}{2}+it)}$ for $N=8$ (a) and $N=9$ (b) (orange) in the range $86 \leq t \leq 90$.}
\label{fig:f5}
	\end{figure}
Figure \ref{fig:f5} shows that $\widetilde{\zeta}_8(s)$ admits two zeros on the critical line in this range. However, $\widetilde{\zeta}_9(s)$ no longer has zeros on the critical line in this range. The reason is that a collision between the two zeros occurred in $\widetilde{\zeta}^t_8(s)$ for a certain $0 \leq t \leq 1$. It should be noted that $\widetilde{\zeta}_9(s)$ has two zeros in the range $86 \leq t \leq 90$ given approximately by $0.74 + 88.12i$ and $0.25 + 88.12i$. The zeros of $\widetilde{\zeta}_N(s)$ remain off the critical until $N=22$ where a second collision occurs and the zeros return to the critical line. This is shown in Fig. \ref{fig:f6} 
	  \begin{figure}[ht!]
	\centering
		\includegraphics[scale=0.35]{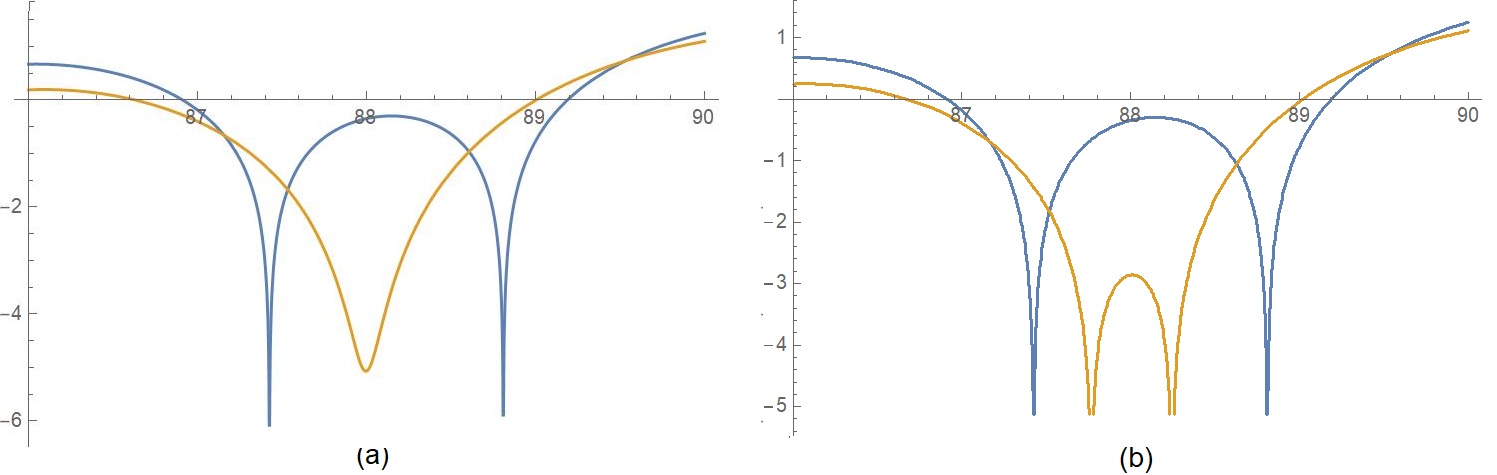} 
	\caption{Graphs of $ln \abs{\zeta(\frac{1}{2}+it)}$ (blue) and
$ln \abs{\widetilde{\zeta}_N(\frac{1}{2}+it)}$ for $N=22$ (a) and $N=23$ (b) (orange) in the range $86 \leq t \leq 90$.}
\label{fig:f6}
	\end{figure}
	
As one can see, even though a collision occurred at $N=8$ and "pushed the zeros off" the critical line, an additional complementing collision occurred at $N=22$ brining the zeros "back to the critical line". 
	\end{ex} 
In general, the RH would follow if this phenomena holds for any consecutive pair of zeros. That is:
\begin{itemize} 
\item The RH would follow if for any pair of zeros the following holds: for any collision "pushing the zeros off" the critical line occurring for certain $N$ there must occur a corresponding collision "pushing the zeros back" to the critical line for later $N'$.   
\end{itemize} 
In fact, we will not consider this question directly. Instead, we will suggest in the next section a method to avoid collisions all together, by changing the order of summation in the elements of $\zeta_N(s)$. Let us conclude this section with the following two remarks:

\begin{rem}[Non resolved collisions for the classical AFE range of approximation]
It should be noted that the sections $\zeta_N(s)$ approximate zeta in two different ways: (a) due to the functional equation it approximates $\zeta(s)$ in the region $\sqrt{2 \pi N} \leq t \leq 2 \pi N$ (which is our main concern). (b) due to the classical approximate functional equation it approximates $\frac{1}{2} \zeta(s)$ in the region $2 \pi N^2 \leq t \leq 2 \pi (N+1)^2$. 

It was already observed by Spira that the sections $\zeta_N(s)$ might admit zeros off the critical line in the region $2 \pi N^2 \leq t \leq 2 \pi (N+1)^2$ (in which according to the AFE it approximates $\frac{1}{2} \zeta(s)$). For instance, for $N=5$ the section $\zeta_5(s)$ admits two zeros off the critical line in the region $219\leq t \leq 222$. The reason for this is that a collision occurs at the previous $N=4$ and is not resolved. It should be noted that the number $N=5$ of sections required for the AFE approximation is extremely small compared to the number of sections $N=\left [ \frac{t}{2} \right]=110$ required for the approximation 
of $\zeta(s)$ in this region, which is the approximation we are interested in.
\end{rem}

\begin{rem}[A new interpretation of Gram's law]
In view of Remark \ref{rem:gram} the elements $t^0_n$ are the zeros of $\abs{1+\chi(s)}$, while the Gram points are actually exactly the local maximum points of the same function on the critical line. Hence, the Gram point could be considered as the "middle point" between two zeros of $1+\chi(s)$. Recall that the $126$-th zero is a zero which violates Gram's law. Figure \ref{fig:f7} shows the way the elements $t^N_{126}=Im(\widetilde{\rho}_{126}^N)$ (orange) of $\widetilde{\zeta}_N(s)$ change for $0 \leq N \leq 89$:  

\begin{figure}[ht!]
		\centering
	\includegraphics[scale=0.9]{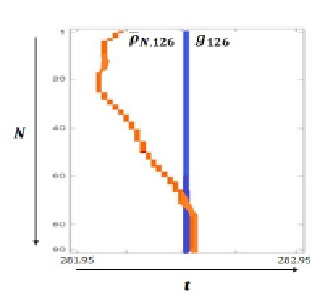}
	\caption{The elements $t^N_{126}=Im(\widetilde{\rho}_{126}^N)$ (orange) and the Gram point $g_{126}$ (blue) for $0 \leq N \leq 89$}
\label{fig:f7}	
	\end{figure}

As one can see in Fig. \ref{fig:f7} during the process of development of the zeros from zero to $N=89$ the imaginary part eventually crosses the Gram point, that is the middle point. Hence, in general, "Gram's law" can be re-phrased as:

\begin{itemize} 
\item Gram's law is the observation that (for each individual $n$) the elements $t^N_n=Im(\widetilde{\rho}^N_n)$ do not tend to eventually cross the Gram point $g_n$, that is, the overall distance travelled by $t^N_n$ is usually less than half the distance between the intial position of the zeros for $N=0$.
\end{itemize}
The following features should be mentioned: 
\begin{enumerate}
\item In Gram's law the point $g_n$ (blue), which is a local maxima of $\abs{\zeta_0(\frac{1}{2}+it) }$, is kept "static" and does not change with respect to $N$. If, however, instead of $g_n$ one considers the local maxima of $ \abs{\widetilde{\zeta}_N(\frac{1}{2}+it)}$ one actually obtains a statement which is essentially our approach to the RH. It should be noticed that the persistence of such local maxima between two consecutive zeros is equivalent to the statement that no collisions occur and their reappearance is the result of the occurrence of a second collision. In this sense, our approach to the RH could be considered as a "dynamic", more refined, version of the cruder "static" Gram's law.    

\item As phrased above Gram's law is a statement about the typical process of development of an \emph{individual} zero, that is of the sequence $t^N_n$ for a \emph{given individual $n$}. Indeed, that the sequence $t_n^N$ for a given individual zero does not tend to pass an overall distance of half the distance between $t^0_{n}$ and $t^0_{n+1}$ (recall formula Eq. \ref{eq:dist}). In our approach, however, we do not consider the behaviour of an individual zero but rather study the \emph{mutual development of a pair of consecutive zeros}. In particular, we are interested in collisions, or lack of collisions, for \emph{pairs} of zeros. In this sense, our approach to the RH could be rephrased as saying that the sequences $t^N_n$ and $t^N_{n+1}$ for a pair of consecutive zeros \emph{can never} mutually pass together an overall distance which exceeds the original distance between $t^0_n$ and $t^0_{n+1}$. In this sense, RH could considered as a "twice stronger" statement than Gram's law.      
 \end{enumerate}
\end{rem}  

Let us consider the following example: 

\begin{ex}[A pair of non-colliding zeros] Figure \ref{fig:f8} shows the consecutive sequences $t^N_{132}$ (brown) and $t^N_{133}$ (blue) together with the imaginary parts zeros $t_{132}$ and $t_{133}$ (left) and the consecutive sequences $\widetilde{t}^N_{132}$ (brown) and $\widetilde{t}^N_{133}$ (blue) together with the imaginary parts of the zeros $\rho_{132}$ and $\rho_{133}$ (right). As one can see from Fig. \ref{fig:f8}, for the presented $132$-th and $133$-th zeros, the sequences $t_{132}^N$ and $t_{133}^N$ (respectively, $\widetilde{t}_{132}^N$ and $\widetilde{t}_{133}^N$) develop separately, and no collisions between the two occurs for any $N$. As a result, the sequences $\rho_{132}^N$ and $\rho_{133}^N$ (respectively, $\widetilde{\rho}_{132}^N$ and $\widetilde{\rho}_{133}^N$) remain on the critical line, for any $N$:
\begin{figure}[ht!]
	\centering
		\includegraphics[scale=0.35]{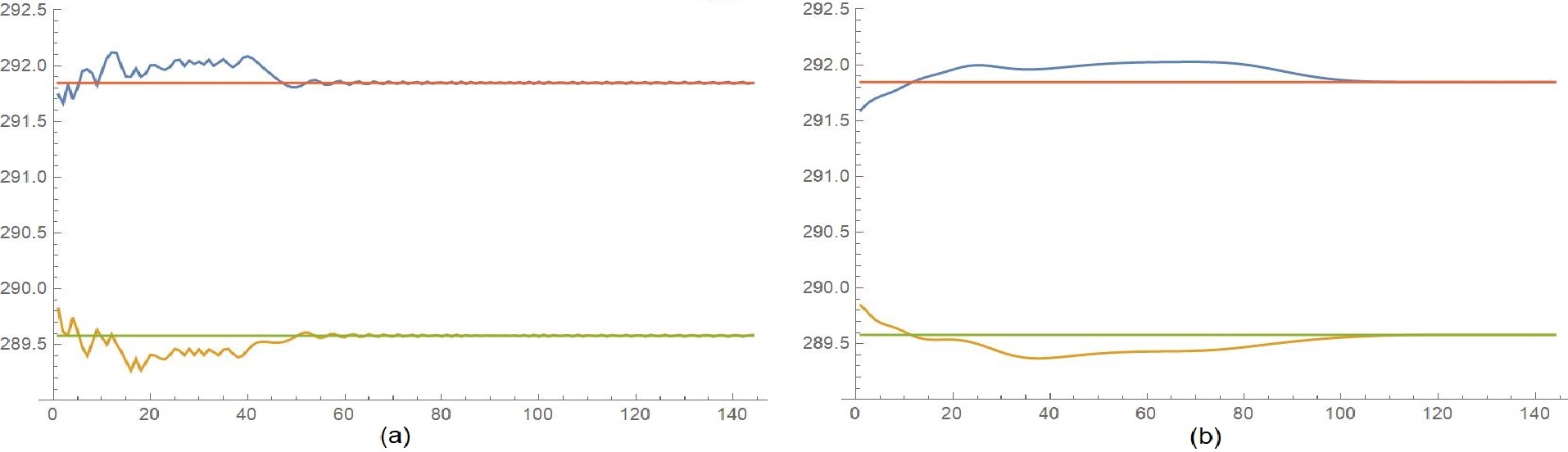} 
	\caption{ (a) The consecutive sequences $t^N_{132}$ (brown) and $t^N_{133}$ (blue) together with $t_{132}$ and $t_{133}$. (b) The consecutive sequences $\widetilde{t}^N_{132}$ (brown) and $\widetilde{t}^N_{133}$ (blue) together with $t_{132}$ and $t_{133}$  for $N=1,...,150$.}
\label{fig:f8}
	\end{figure}

 It is obvious that for any such pair of consecutive zeros $\rho_n$ and $\rho_{n+1}$, for which a collision does not occur in their corresponding sequences, both zeros must remain on the critical line. That is, any such pair of non-colliding zeros nessecerally satisfy RH. In the next section we would consider the more "interesting" case of colliding zeros. 

It should also be noted that both the classical $\rho_n^N$ and transformed $\widetilde{\rho}_n^N$ sequences are non-colliding in the case of $n=132,133$ presented in Fig. \ref{fig:f8}. This is typical of the general case in which the classical and transformed sequences are seen to collide$\setminus$non-collide together. That is the collision property is an essential feature of a given pair of zeros and not of the specific sequence (classical or transformed) considered.
 \end{ex}
 
  \section{Avoiding collisions via the repelling re-arrangement of summation} 
 \label{s:4}  
 
 In the previous section we saw an example of non-colliding pair of zeros. Before explaining how collisions might be avoided let us first consider an example of colliding zeros. 

\begin{ex}[A pair of colliding zeros] The following Fig. \ref{fig:f9} shows the consecutive sequences $t^N_{725}$ (brown) and $t^N_{726}$ (blue) together with the imaginary parts of the zeros $t_{725}$ and $t_{726}$ (left) and the consecutive sequences $\widetilde{t}^N_{725}$ (brown) and $\widetilde{t}^N_{726}$ (blue) together with the imaginary parts of the zeros $\rho_{725}$ and $\rho_{726}$ (right) for $N=1,..,300$:
\begin{figure}[ht!]
	\centering
		\includegraphics[scale=0.35]{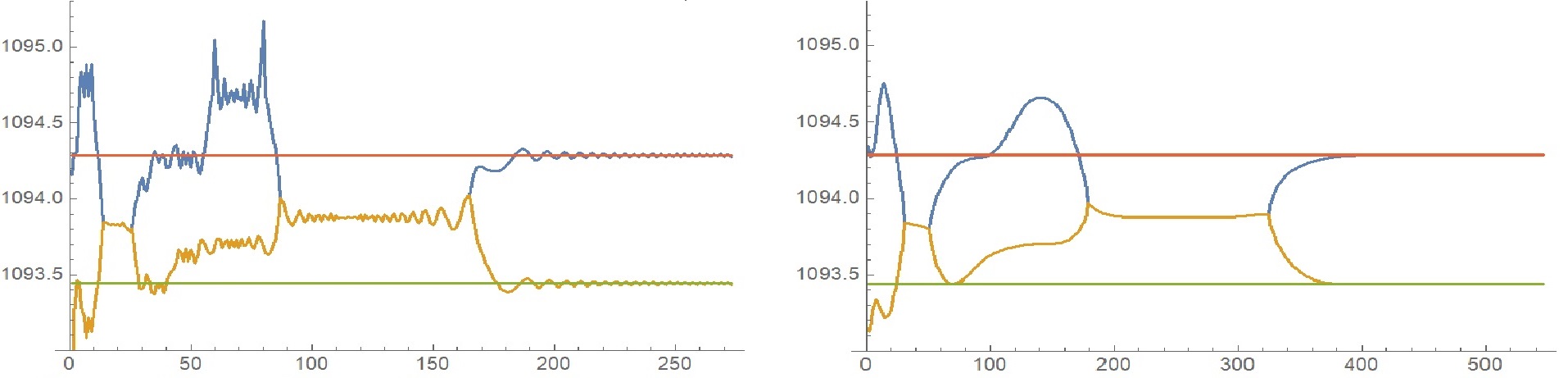} 
	\caption{ (a) The consecutive sequences $t^N_{725}$ (brown) and $t^N_{726}$ (blue) with $t_{725}$ (green) and $t_{726}$ (red). (b) The consecutive sequences $\widetilde{t}^N_{725}$ (brown) and $\widetilde{t}^N_{726}$ (blue) with $t_{725}$ (green) and $t_{726}$ (red) for $N=1,...,300$.}
\label{fig:f9}
	\end{figure}

As one can see from Fig. \ref{fig:f9}, for the presented $725$-th and $726$-th zeros, the sequences $t_{725}^N$ and $t_{726}^N$ (respectively, $\widetilde{t}_{725}^N$ and $\widetilde{t}_{726}^N$) do collide for certain values of $N$. However, as expected, after this collision the two zeros eventually part and continue on separate paths.   

Note that the development of the sequences $t_n^N$ and $\widetilde{t}^N_n$ bare some clear resemblance to the partial sums $S_N(s)$ and $\widetilde{S}_N(s)$, as discussed in Section \ref{s:1}, compare for instance Fig. \ref{fig:f2}. Indeed, we see that the sequences $t_n^N$ and $\widetilde{t}^N_n$ fluctuate around various values before stabilizing at the limit $t_n$ for $N>>0$. In fact, like in the case of the partial sums $S_N(s)$ the transitions in the sequence $t_n^N$ occur in the regions $\left [ \frac{t}{2 (M+1) \pi } \right ] \leq t \leq \left [ \frac{t}{2 M \pi } \right ]$. For instance the two collisions for $t_n^N$ in Fig. \ref{fig:f9} occur at the following intervals: 
 \begin{enumerate}
 \item The first collision occurs around the interval $\left [ \frac{t}{24 \pi } \right ] \leq N \leq \left [ \frac{t}{12 \pi } \right ]$
  \item The second collision occurs around the interval $\left [ \frac{t}{4 \pi } \right ] \leq N \leq \left [ \frac{t}{2 \pi } \right ]$. 
 \end{enumerate} 
We expect that up until around the half of these intervals the two zeros are "attracted towards each other" and afterwards, in the second half of the interval, get "repelled away from each other". According to this viewpoint, let us consider the following 
repelling re-arrangements of the indices 
\begin{equation} 
\begin{array}{ccc}
R(n):= \left \{ \begin{array}{cc} n & n \leq 13 \\ 41-n & 13<n<28 \\ n & 28 \leq n \leq 85 \\ 256 -n & 85<n<171 \\ n & 171 \leq n \end{array}  \right. & ; & 
\widetilde{R}(n):= \left \{ \begin{array}{cc} n & n \leq 30 \\ 84-n & 30<n<54 \\ n & 54 \leq n \leq 177 \\ 504 -n & 177<n<327 \\ n & 327 \leq n \end{array}  \right. .
\end{array}
\end{equation}   
The idea behind defining $R(n)$ and $\widetilde{R}(n)$ in such a manner is that we want to add the "repelling" elements before "attracting" ones out of an expectation that this would lead to a cancellation of the collisions. In particular, let us define the re-arranged sections 
\begin{equation} 
\begin{array}{ccc} \zeta_N(s;R) := \sum_{n=1}^{N} B_{R(n)} (s) & ; & \widetilde{\zeta}_N(s; \widetilde{R}) := \sum_{n=1}^{N} \widetilde{B}_{\widetilde{R}(n)} (s). \end{array} 
\end{equation}    
Set $\rho_n^N(R)$ and $\widetilde{\rho}_n^N(\widetilde{R})$ for the zeros of $\zeta_N(s;R)$ and $\widetilde{\zeta}_N(s;\widetilde{R})$ and denote by $t_n^N(R)$ and $\widetilde{t}_n^N(\widetilde{R})$ the corresponding imaginary parts. Consider Fig. \ref{fig:f10}:  
 \begin{figure}[ht!]
	\centering
		\includegraphics[scale=0.35]{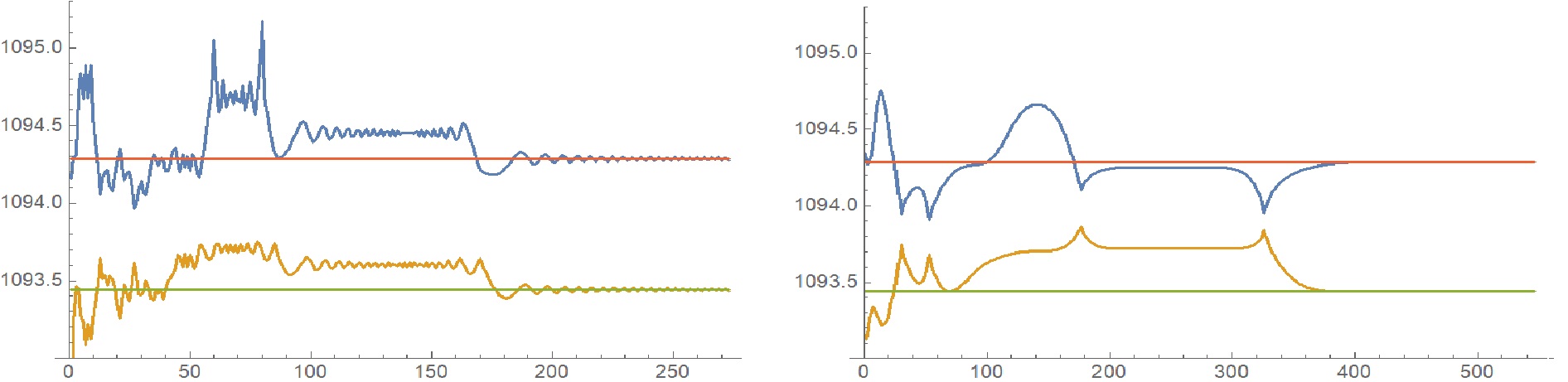} 
	\caption{ (a) The rearranged sequences $t^N_{725}(R)$ (brown) and $t^N_{726}(R)$ (blue) with $t_{725}$ (green) and $t_{726}$ (red). (b) The rearranged sequences $\widetilde{t}^N_{725}(\widetilde{R})$ (brown) and $\widetilde{t}^N_{726}(\widetilde{R})$ (blue)  with $t_{725}$ (green) and $t_{726}$ (red) for $N=1,...,300$.}
\label{fig:f10}
	\end{figure}
	
	Figure \ref{fig:f10} shows us is that, contrary to Fig. \ref{fig:f9} for the ordinary sections, for the rearranged sections $\zeta_N(s;R)$ and $\widetilde{\zeta}_N(s;\widetilde{R})$ the collisions have been altogether avoided. In other words, the collisions seen in Fig. \ref{fig:f9} are not an essential feature of the zeros but rather a by-product of the trivial order of summation considered. Of course, whenever a collision can be cancelled it implies that the pair of zeros satisfy RH. 
 \end{ex}
 We conjecture that the procedure described above holds in general and can be applied to any consecutive pair of zeros. Concretely, we conjecture:
	\begin{itemize} 
\item For any pair of consecutive zeros, collisions in the sequences $t_n^N$ and $t_{n+1}^N$ (or $\widetilde{t}_n^N$ and $\widetilde{t}_{n+1}^N$) can occur only along intervals of the form $\left [ \frac{t}{2 M_1 \pi } \right ] \leq N \leq \left [ \frac{t}{2 M_2 \pi } \right ]$. A collision can always be cancelled by rearranging the order of summation of the elements $B_n(s)$ ($\widetilde{B}_n(s)$) to be in reverse order along such an interval of collision.
\end{itemize}
	
	Let us conclude this section with the following remarks:
	
	\begin{rem}[The chaotic region] Let us note that intervals of the form 
	\begin{equation} \label{eq:int} \left [ \frac{t}{2 (M+1) \pi } \right ] \leq N \leq \left [ \frac{t}{2 M \pi } \right ]
	\end{equation} become smaller as $M$ grows and after some $M_0=M_0(t)$ become essentially trivial. Thus the range of existence of the regulated intervals only begins after an initial region of the form $N \leq \left [ \frac{t}{2 M_0 \pi } \right ]$, to which we refer as the "chaotic region" of $t$. In particular, our conjecture includes the assumption that collisions cannot occur at all during the initial chaotic region but rather only when the intervals Eq. \ref{eq:int} start becoming regulated, that is, big enough. We can rephrase by saying that the initial chaotic region does not have enough "energy" to bring the zeros into collision or to travel a mutual distance of more than $\abs{t^0_{n+1}-t^0_n}$, compare Eq. \ref{eq:dist}.  
		\end{rem} 
  \begin{rem}[Non-cancellation of collision for the Davenport-Heilbronn function] \label{rem:DH} Recall that the Davenport-Heilbronn functions are a class of Dirichlet functions which satisfy a functional equation but for which RH fails, that is for which there exist zeros off the critical line, see \cite{DH,T}. Consider the function 
\begin{equation}
\mathcal{D}(s)= \frac{(1-i \kappa)}{2} L(s,\chi_{5,2}) + \frac{(1+i \kappa)}{2} L(s,\overline{\chi}_{5,2})
\end{equation}
with 
\begin{equation} 
\kappa=\frac{\sqrt{10-2\sqrt{5}}-2}{\sqrt{5}-1}.
\end{equation}
The functional equation for the Davenport-Heilbronn function $\mathcal{D}(s)$ is given by 
\begin{equation} 
\xi(s)=\xi(1-s), 
\end{equation}
where 
\begin{equation} 
\xi(s) = \left ( \frac{\pi}{5} \right )^{-\frac{s}{2}} \Gamma \left ( \frac{1+s}{2} \right ) \mathcal{D}(s).  
\end{equation} 
It should be noted that contrary to $\zeta(s)$ the function $\mathcal{D}(s)$ does not have an Euler product. If we apply
Euler acceleration of series to $\mathcal{D}(s)$ we can express 
\begin{equation} 
\mathcal{D}(s) = \sum_{n=0}^{\infty} \widetilde{A}_{DH}(n,s), 
\end{equation} 
where 
\begin{equation} 
\widetilde{A}_{DH}(n,s):= \frac{1}{2^{n+2}} \sum_{k=0}^n (-1)^{k+1} \binom{n}{k}  \frac{(1-i \kappa)\chi_{5,2}(k)+(1+i \kappa)\overline{\chi}_{5,2}(k) }{k^s}.
\end{equation} 
In view of the functional equation, we define 
\begin{equation} 
\widetilde{\xi}_N(s)= \sum_{n=0}^N \widetilde{B}_{DH}(n,s),
\end{equation} 
where 
\begin{equation} 
\widetilde{B}_{DH}(n,s):=\frac{1}{2} \left [  \left ( \frac{\pi}{5} \right )^{-\frac{s}{2}} \Gamma \left ( \frac{1+s}{2} \right ) \widetilde{A}_{DH}(n,s) + \left ( \frac{\pi}{5} \right )^{\frac{s-1}{2}} \Gamma \left ( \frac{2-s}{2} \right ) \widetilde{A}_{DH}(n,1-s) \right ] 
\end{equation} 
We mention (without proof) that $\widetilde{X}_N(s)$ attains exponential accuracy of $\xi(s)$ around $N=N(s)=\left [ 2 t \right ]$ in this case. 

Franca and LeClair show in \cite{FL,FL2} that the zeros $\rho^{DH}_{0,n}= \frac{1}{2}+i t^{DH}_{0,n}$ of $\widetilde{\xi}_0$ all lie on the critical line their imaginary part $t^{DH}_{0,n}$ can be expressed in terms of the Lambert function as follows 
\begin{equation} 
t^{DH}_{0,n}:=\frac{2 \pi (n - \frac{5}{8} ) }{W_0 (5 e^{-1}(n-\frac{5}{8} ))}. 
\end{equation}
The following Fig. \ref{fig:f11} shows (a) the consecutive sequences $t^{DH}_{N,44}$ (brown) and $t^{DH}_{N,45}$ (blue) with $t^{DH}_{44}$ (green) and $t^{DH}_{45}$ (red). (b) The consecutive sequences $\widetilde{t}^{DH}_{N,44}$ (brown) and $\widetilde{t}^{DH}_{N,45}$ (blue) with $t^{DH}_{44}$ (green) and $t^{DH}_{45}$ (red) for $N=1,...,200$.
\begin{figure}[ht!]
	\centering
		\includegraphics[scale=0.35]{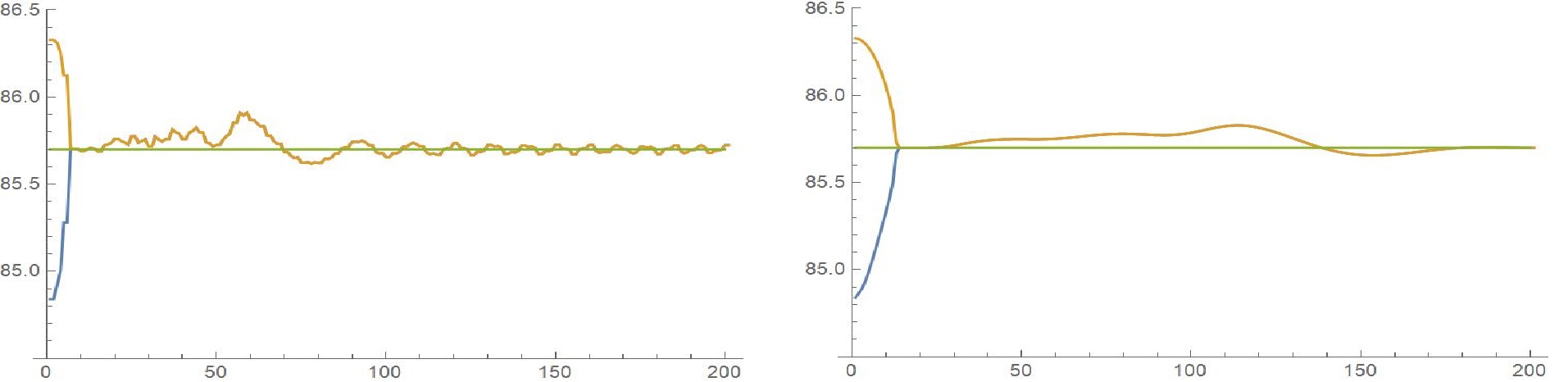} 
	\caption{ (a) The consecutive sequences $t^{DH}_{N,44}$ (brown) and $t^{DH}_{N,45}$ (blue) with $t^{DH}_{44}$ (green) and $t^{DH}_{45}$ (red). (b) The consecutive sequences $\widetilde{t}^{DH}_{N,44}$ (brown) and $\widetilde{t}^{DH}_{N,45}$ (blue) with $t^{DH}_{44}$ (green) and $t^{DH}_{45}$ (red) for $N=1,...,200$.}
\label{fig:f11}
	\end{figure}
	
In contrast to the zeta function, the the collision of the sequences presented in Fig. \ref{fig:f11} is essential and cannot be cancelled by changing the order of summation. Moreover, the collision occurs at $N=12$, that is, already within the initial \emph{chaotic region}. In other words, it seems that in the Davenport-Heilbronn setting the chaotic region does have enough \emph{energy$\setminus$momentum} to collide consecutive zeros. In particular, the zeros of the DH function $\mathcal{D}(s)$ off the critical line seem to present an essentially different behaviour then those of the zeta function $\zeta(s)$.  
\end{rem} 

\begin{rem}[Montgomery's pair correlation conjecture] \label{rem:MPC} For pairs of zeros of zeta one has (assuming RH) the famous pair correlation conjecture 
due to Montgomery, see \cite{M}. Let $\alpha \leq \beta$ and set 
\begin{equation} 
A(T ; \alpha, \beta):= \left \{ (\rho,\rho') \mid  0 < \rho,\rho'<T \textrm{ and } \frac{2 \pi \alpha}{ln(T)} \leq \rho - \rho' \leq \frac{2 \pi \beta}{ln(T)} \right \}. \ 
\end{equation}  
The conjecture states that 
\begin{equation} 
N(T ; \alpha,\beta):= \sum_A 1 \sim \left (\int_{\alpha}^{\beta} \left (1- \frac{sin (\pi u)}{\pi u} \right )du + \delta_0([\alpha,\beta]) \right ) \frac{T}{2 \pi} ln(T) 
\end{equation} 
for $T \rightarrow \infty$. Since the integral is small when $u$ is small, the conjecture is typically intuitively understood as expressing the idea that consecutive zeros "repel each other". Hence, although the term repulsion suggests a dynamic relation, in the context of the Montgomery conjecture the notion of "repulsion" is a statistical one. 

In our setting we also speak about "repulsion" between consecutive zeros of zeta or, more concretely, between the consecutive sequences $t_n^N$ and $t_{n+1}^N$. However, in our setting the notion of repulsion is indeed a dynamic one rather than statistic. Concretely, we say that at stage $N$ the sequence is "repelling" if the distance between the two sequences decreases 
\begin{equation} 
\abs{t_{n+1}^N-t_n^N}<\abs{t_{n+1}^{N+1}-t_n^{N+1}}
\end{equation} 
and "attracting" if the opposite inequality occurs and the zeros get closer to one another. Note that the repulsion or attraction of adding $B_n(s)$ depends on the position of the zeros of $B_n(s)$ (described in Proposition \ref{prop:2.1}) relative to the zeros of the section to which they are added. Our approach is that for any pair of consecutive zeros the overall "attraction" cannot exceed the original distance $\abs{t^0_{n+1}-t^0_n}$. 
\end{rem}
 \section{Summary and Concluding Remarks} 
 \label{s:5} 
 
The Riemann hypothesis is the postulate that all non-trivial zeros of $\zeta(s)$ lie on the critical line. Riemann himself computed the first three zeros of $\zeta(s)$ (as communicated by Siegel \cite{S}). As of today, the Riemann hypothesis has been numerically verified for zeros $\rho$ with $Im(\rho)$ up to around $3 \cdot 10^{12}$, see \cite{PT}. However, since its introduction a plausibility argument for RH, that is a conceptual reason for its validity, aside from numerical verification, has been essentially thought-after\footnote{For instance Edwards writes as follows in his classical book \cite{E}: \emph{"Even today, more than a hundred years later, one cannot really give any solid reasons for saying that the truth of the RH is "probable" etc. Also the verification of the hypothesis for the first three and a half million roots above the real axis perhaps makes it more "probable". However, any real \bf \emph{reason} \rm, any plausibility argument or heuristic basis for the statement, seems entirely lacking" (H. M. Edwards, 1974)}}.

In this work we suggested a new approach for the study of the zeros of the zeta function $\zeta(s)$ by studying the dynamic way in which the zeros of the section $\zeta_N(s)$ change with respect to $N$. The zeros of $\zeta_1(s)$, which are regulated and well-understood, start on the critical line and then move their position as $N$ changes. A pair of zeros can go off the critical line only if a collision between the two zeros occur. We conjectured, based on numerical evidence, that collisions could always be avoided by changing the order of summation of the elements $B_n(s)$ comprising the section $\zeta_{\left [ \frac{t}{2} \right ]}(s)$, in a specific manner to which we refer as a repelling re-arrangement. 

In this sense we suggest that the observed dynamic repulsion phenomena could be viewed as a certain "plausibility argument" for RH. Indeed, the idea suggested in this work, that the zeros begin on the critical line and then develop in a way that dynamically repels them for colliding, obliges them to stay on the critical line, which at the end of the process is essentially the RH. Of course, as explained, if it would be possible to prove that the introduced repelling re-arrangements indeed avoid collisions, in general, for any pair of consecutive zeros, this would imply RH.

\end{document}